\documentclass[12pt]{amsart}
\usepackage{amsmath}

\def\lf{\left}
\def\ri{\right}
\def\a{{\alpha}}

\def\wt{\widetilde}

\def\p{\partial}
\newcommand\R{{\mathbb R}}

\newcommand\C{{\mathbb C}}
\newcommand\tM{{\tilde M}}

\def\ii{\sqrt{-1}}
\def\jbar{{\bar\jmath}}

\def\K{K\"ahler }

\def\A{Amp\`{e}re }

\def\be{\begin{equation}}
\def\ee{\end{equation}}

\def\lf{\left}
\def\ri{\right}
\def\a{{\alpha}}

\def\ijb{{i\jbar}}

\def\Ric{\text{\rm Ric}}

\def\wt{\widetilde}

\def\p{\partial}

\def\C{\Bbb C}

\def\wt{\widetilde}

\def\p{\partial}

\def\p{\partial}

\def\C{\Bbb C}
\def\ii{\sqrt{-1}}

\def\lf{\left}
\def\ri{\right}
\def\a{{\alpha}}

\def\ijb{{i\jbar}}

\def\Ric{\text{\rm Ric}}

\def\wt{\widetilde}

\def\p{\partial}

\def\C{\Bbb C}

\def\wt{\widetilde}

\def\p{\partial}

\def\p{\partial}

\def\C{\Bbb C}
\def\ii{\sqrt{-1}}

\def\tM{\wt M}
\def\vol{\text{Vol}}

\def\t{\tilde}

\newtheorem{thm}{Theorem}[section]

\newtheorem{lem}{Lemma}[section]

\newtheorem{cor}{Corollary}[section]
\theoremstyle{definition}
\newtheorem{defn}{Definition}[section]
\theoremstyle{remark}

\numberwithin{equation}{section}
\begin{document}
\title{On quadratic orthogonal bisectional curvature}

\author{Albert Chau$^1$}

\address{Department of Mathematics,
The University of British Columbia, Room 121, 1984 Mathematics
Road, Vancouver, B.C., Canada V6T 1Z2} \email{chau@math.ubc.ca}

\author{Luen-Fai Tam$^2$}
\address{The Institute of Mathematical Sciences and Department of
 Mathematics, The Chinese University of Hong Kong,
Shatin, Hong Kong, China.} \email{lftam@math.cuhk.edu.hk}
\thanks{$^1$Research
partially supported by NSERC grant no. \#327637-11}
\thanks{$^2$Research partially supported by Hong Kong  RGC General Research Fund
\#CUHK 403011}

\begin{abstract}
In this article we study compact \K manifolds satisfying a certain nonnegativity condition on the bisectional curvature.  Under this condition, we show that the scalar curvature is nonnegative and that the first Chern class is positive assuming local irreducibility.  We also obtain a partial classification of possible de Rham decompositions of the universal cover under this condition.  \end{abstract}

\maketitle
\markboth{Albert Chau and Luen-Fai Tam} {Quadratic orthogonal bisectional curvature}

 \section{Introduction}

We begin with the following definition.
  \begin{defn}\label{def-QOBC} A \K manifold  $(M, g)$ of complex dimension $n$ is said to have {\it nonnegative  quadratic orthogonal   bisectional  curvature (NQOBC)} at $p\in M$ if $n\geq 2$ and: {\it for any} unitary frame $\{e_1,\dots,e_n\}$ of $T^{(1,0)}_p(M)$ and {\it any} real numbers $\xi_1,...,\xi_n$ we have
\begin{equation}\label{e1}
\sum_{i,j=1}^n R_{i\bar{i}j\bar{j}} (\xi_i -\xi_j)^2 \geq 0.
\end{equation}
We will say that a manifold $(M, g)$ has NQOBC provided it does at every point.
\end{defn}
One can similarly define the notion of nonpositive orthogonal quadratic  bisectional  curvature. Note that if a product of \K manifolds $M_1 \times M_2$ has NQOBC then so must each factor $M_1$ and $M_2$. The reverse implication however may be false in general:  $M_1$ and $M_2$ may both have NQOBC while $M_1 \times M_2$ may not.

The condition of NQOBC is weaker than requiring $M$ to have nonnegative orthogonal bisectional curvature: $R(V,\bar V,W,\bar W)\geq 0$ for any orthogonal unitary pair $V, W\in T^{(1,0)}(M)$, while on complex surfaces the two conditions are equivalent.    Li, Wu and Zheng  \cite{LiWuZheng} are able to construct examples of compact \K
manifolds $M$ with NQOBC which do not admit
any \K metrics with nonnegative orthogonal bisectional curvature.

 The structure of compact \K manifolds with nonnegative orthogonal
bisectional curvature on the other hand, is now completely understood by
the works of Chen \cite{Chen} and Gu and Zhang \cite{GZ}. Their results
extend those on the generalized Frankel conjecture for compact manifolds
with nonnegative holomorphic bisectional curvature, formulated by Yau \cite[p.677] {Yau}, and
established through the works of Howard-Smyth-Wu \cite{HowardSmythWu1981}, Wu \cite{Wu} and eventually by Bando \cite{Bando} in three-dimension and Mok \cite{Mok} for all dimensions. Frankel's
conjecture for compact manifolds with positive holomorphic bisectional
curvature was established by Mori \cite{M}   and Siu and Yau \cite{SY} independently.

The NQOBC condition arises naturally in the proof of the following fact which appeared implicitly in an earlier work \cite{BG}:
  \begin{lem}\label{l-keyfact}
  If $(M^n, g)$ has NQOBC, then all harmonic $(1,1)$ forms are parallel.
  \end{lem}
The proof uses the Bochner formula for $(1,1)$ forms.  The lemma is implicit from earlier works though was only used in these under the stronger assumptions of positive and nonnegative holomorphic bisectional curvature  (see for example  \cite{BG, GK, HowardSmythWu1981} and references therein).  We will make use of Lemma \ref{l-keyfact} at various points.

 The NQOBC condition was first considered explicitly by Wu-Yau-Zheng in \cite{WuYauZheng} where the authors studied the boundary of the \K cone of manifolds with NQOBC.

 In this paper, we try to understand more on  this class of \K manifolds.
We first prove that: {\it If a \K manifold has NQOBC at a point $p$, then the scalar
curvature is nonnegative at $p$ and is zero if and only if it is flat at $p$, provided
the complex dimension is at least 3}. See Theorem \ref{t-scalarcurvature} for more details.
Using this result we prove: {\it If $(M,g)$ is a compact \K manifold with NQOBC which is locally irreducible, then $c_1(M)>0$}. This generalizes \cite[Theorem 2.1]{GZ}.  This naturally leads us to consider the deRham factorization of the universal cover of $M$ for which we prove the following. {\it All compact factors have positive first Chern class.  There is at most one non-compact factor being either: complex Euclidean space (in which case all compact factors will have quasi positive Ricci curvature),  or 1-dimensional (in which case all  compact factors will have positive Ricci curvature)}.  This gives a slight refinement of \cite[Theorem 1.3 (2)]{GZ}.  For more detailed description, see Theorem \ref{t-reduce}.

The organization of the paper is as follows. In Section \ref{s-cone}, we will give a quick proof of a result in \cite{WuYauZheng} on the boundary of the \K cone of manifolds with NQOBC which is similar to but independent of the   proof  in \cite{Z} and the idea is similar. In Section \ref{scalar}, we will prove that \K manifolds with NQOBC must have nonnegative scalar curvature.  In Section 4 we prove the above mentioned structure results for compact \K manifolds with NQOBC.

 The authors would like to thank D. Wu, S.T. Yau and  F. Zheng for helpful comments and interest in this work.

\section{Boundary of the \K cone of manifolds with NQOBC}\label{s-cone}

The boundary of the \K cone of $(M, g)$ turns out to be rather special: every boundary class $\alpha$ contains a smooth nonnegative representative $\eta \in \alpha$.  This was first proved in \cite{WuYauZheng} by solving a degenerate complex Monge-\A equation under the condition of NQOBC. This motivates us to study structure of \K manifolds with NQOBC.   In fact, the following holds:

\begin{thm}\label{s2t1}
Suppose $(M^n, g)$ is a compact complex n-dimensional \K manifold with NQOBC.  Let $\alpha$ be in the closure of the \K cone of $M$ and $\eta$ be the unique harmonic representative in $\alpha$.  Then $\eta$ is nonnegative.  Moreover, $\eta$ is positive if and only if $\alpha^n[M] >0$.
\end{thm}

A short proof of Theorem \ref{s2t1} was provided in \cite{Z} and also independently in \cite{CT}.  We provide details of the proof here for the sake of  completeness.  We follow the presentation in \cite{CT} which we think is more concise.      We refer to \cite{CT} where it was also shown that a rather straight forward observation on the proof in \cite{WuYauZheng} leads to another proof of Theorem \ref{s2t1}.

 Given a complex manifold $M$, recall that a real class $\alpha\in H^{(1,1)}(M)$ is called a \K class if $\alpha$ contains a smooth positive definite representative $\eta$.  The space of \K classes is a convex cone in $H^{(1,1)}(M)$  referred to as the \K cone which we denote by $\mathcal{K}$.  We say that $\alpha$ is in the closure of $\mathcal{K}$ if $[(1-t)\omega+t\eta] \in\mathcal{K}$ for any smooth $\eta\in\alpha$, $\omega\in\mathcal{K}$ and $t\in[0,1)$.  Finally, given any real $\alpha\in H^{(1,1)}(M)$ we use $\alpha^n[M]$ to denote the integral $\int_M  \alpha^n$.

\begin{proof} [Proof of Theorem \ref{s2t1}]
 Let $\eta$ be as in Theorem \ref{s2t1} and let $\omega_0$ be the \K form for $(M, g)$.  By the above remarks, $\eta$ is parallel and thus has constant real eigenvalues $a_1,...,a_n$ on $M$ with respect to $\omega_0$.  Also, $[(1-t)\omega_0+t\eta]\in \mathcal{K}$  for every $t\in[0,1)$.

  In other words, for each $t\in[0,1)$ there exists $f_t\in C^{\infty}(M)$ and $\omega_t \in \mathcal{K}$ such that $(1-t)\omega_0+t\eta=\omega_t +dd^c f_t$, giving
$$
\vol_g (M)\prod_{i=1}^n(1-t+ta_i)= \int_M((1-t)\omega_0+t\eta)^n=\int_M\lf(\omega_{t}+dd^c f_{t}\ri)^n>0
$$
for all $t\in[0,1)$.  On the other hand, if $a_k<0$ for some $k$ then $1-t+ta_k$ and thus the product on the LHS above would vanish for some $t_0\in(0,1)$ giving a contradiction. Thus $a_i$ must be nonnegative for each $i$, in other words $\eta$ is nonnegative.  In particular, we have $\int_M \eta^n \geq  0$ with strict inequality if and only if $\eta$  is positive.  This completes the proof.
\end{proof}

We end this section with the following description of the boundary of $\mathcal{K}$ from \cite{CT}.  Let $\wt M$ be the universal cover of $M$ with projection $\pi:\wt M\to M$.  Then by the de Rham decomposition Theorem for \K manifolds, we may write $$(\wt M, \wt \omega_0)=(\wt M_0, \wt \sigma_0)\times (\wt M_1, \wt \sigma_1)\times\cdots\times (\wt M_k , \wt \sigma_k)$$ where $\wt \omega_0=\pi^*(\omega_0)$, $\tM_0$ is flat and  each $\tM_i$, for $i\ge 1$,  is nonflat irreducible and \K and the decomposition is unique up to permutation.

 \begin{cor}\label{c1} Let $\eta$ be a real harmonic (1,1) form on $M$, and let $\wt \eta=\pi^*(\eta)$. Then $\wt\eta=\wt \eta_0\times\prod_{i=1}^k a_i\wt \sigma_i$ for some real constants $a_i$ and $\wt\eta_0$ is a parallel harmonic form on $\tM_0$. In particular, the boundary of $\mathcal{K}$ can be identified with the space of harmonic $(1,1)$ forms $\wt \eta$ on $\wt M$ satisfying: $\wt \eta$ is equivariant with respect to $\pi_1(M)$ and $\wt\eta=\wt\eta_0\times\prod_{i=0}^k a_i\wt \sigma_i$, where $\wt \eta_0\ge0$ and $a_i\ge0$ for all $i$ with equality holding for some $i$.
\end{cor}
\begin{proof} Since $\eta$ is parallel, then $\wt M$ is a product of \K manifolds $\wt N_i$ such that at each point $T^{(1,0)}(N_i)$ are eigenspace of $\wt \eta$. By the uniqueness and irreducibility of the $\tM_i$'s for $i\ge 1$, the result follows.
\end{proof}

\section{Scalar curvature of manifolds with NQOBC}\label{scalar}
Let $(M^n,g)$ be a \K manifold. If $M$ has nonnegative holomorphic bisectional curvature, then the Ricci curvature of $M$ is nonnegative. If on the other hand $M$ only has NQOBC, then the Ricci curvature may have negative eigenvalues at some point, see \cite[Example 1.2]{GZ}. However, we have the following:

\begin{thm}\label{t-scalarcurvature} Suppose $(M^n,g)$ has nonnegative (resp. nonpositive) QOBC at $p$.  Then the scalar curvature $S(p)$ is nonnegative (resp. nonpositive) and $S(p)=0$ if and only if for all unitary pairs $V,W$ we have $R(V,\bar V,W,\bar W)=0$ and $R(V,\bar V,V,\bar V)+R(W,\bar W,W,\bar W)=0$. If $n\ge 3$, $S(p)=0$ also implies $R(V,\bar V,V,\bar V)=0$ for all $V\in T_p^{(1,0)}(M)$ and hence $M$ is flat at $p$.
\end{thm}
The nonnegativity (nonpositivity) of the scalar curvature $S(p)$ in the theorem follows from the following slightly more general result.
\begin{lem}\label{l-scalar} Let $(M^n,g)$ be a \K manifold and $p\in M$.
 \begin{itemize}
   \item [(a)] Suppose there exists an $n\times n$ matrix $a_{ij}$ with the properties: $\sum_{i\neq j}a_{ij}+2\sum_{i}a_{ii}>0$  and {\it for any} unitary frame $\{e_1,\dots,e_n\}$ of $T^{(1,0)}_p(M)$,
\begin{equation}\label{e-weaker1}
\sum_{i,j=1}^n R_{i\bar{i}j\bar{j}} a_{ij} \geq 0.
\end{equation}
Then $S(p)$ is nonnegative. If in addition, there is a unitary frame such that \eqref{e-weaker1} is a strict inequality  then $S(p)>0$.
   \item [(b)] Suppose there exists an $n\times n$ matrix  $a_{ij}$ with the properties: $\sum_{i\neq j}a_{ij}+2\sum_{i}a_{ii}>0$ and {\it for any} unitary frame $\{e_1,\dots,e_n\}$ of $T^{(1,0)}_p(M)$,
\begin{equation}\label{e-weaker2}
\sum_{i,j=1}^n R_{i\bar{i}j\bar{j}} a_{ij} \le 0.
\end{equation}
Then $S(p)$ is nonpositive. If in addition, there is a unitary frame such that \eqref{e-weaker2} is a strict inequality, then $S(p)<0$.
 \end{itemize}
\end{lem}
\begin{proof} To prove (a), let $\mathcal{F}$ be the set of all ordered unitary bases in $T^{(1,0)}_p(M)$.  For any frame $\a=\{e_1,\dots,e_n\}\in \mathcal{F}$ and $\mathfrak{u}=(u_{ij})\in U(n)$, we define the frame $$\mathfrak{u}\a:=\lf\{\sum_{j}u_{j1}e_j,\dots,\sum_{j}u_{jn}e_j\ri\}\in \mathcal{F}.$$   In other words, $\mathfrak{u}$ is just the change of basis matrix from $\a$ to $\mathfrak{u}\a$.  Now we choose some frame $\alpha' \in \mathcal{F}$ to be fixed throughout the proof.

 Define functions $F(ij)$ (for any $i,j$) and $G(i)$ (for any $i$) on $U(n)$ by
\begin{equation}
\begin{array}{clcr}
			&F(ij)(\mathfrak{u}):=R(f_i,\bar f_i,f_j,\bar f_j)\\
			&G(i)(\mathfrak{u}):=R(f_i,\bar f_i,f_i,\bar f_i)\\
		\end{array}
\end{equation}
where $\{f_1,\dots,f_n\}=\mathfrak{u}\a'$.  Now we establish the following

\begin{bf} Claim:\end{bf} Let $\mu$ be the left invariant Haar measure on $U(n)$.  Then we  have $2\int_{U(n)}F(ij) (\mathfrak{u})d\nu(\mathfrak{u})=\int_{U(n)} G(k)(\mathfrak{u})d\nu$
for any $i\neq j$ and for any $k$.  In particular, the LHS is independent of $i,j$ and the RHS is independent of $k$.

For $i\neq j$, let $\mathfrak{u_0}(ij)\in U(n)$ be the unitary matrix satisfying: for any $\{f_1,\dots,f_n\}\in \mathcal{F}$, $\{h_1,\dots,h_n\}=\mathfrak{u_0}(ij)\{f_1,\dots,f_n\}$ satisfies
\begin{equation}
\begin{array}{clcr}
			h_i=\frac{ 1+\ii}2f_i+\frac{ 1-\ii}2f_j\\
			h_j=\frac{ 1-\ii}2f_i+\frac{ 1+\ii}2f_j\\
		\end{array}
\end{equation}
and $h_k=f_k$ if $k\neq i, j$.  Also let $\mathfrak{v_0}(ij)\in U(n)$ be the unitary matrix satisfying: for any $\{f_1,\dots,f_n\}\in \mathcal{F}$, $\{h_1,\dots,h_n\}=\mathfrak{v_0}(ij)\{f_1,\dots,f_n\}$ satisfies
\begin{equation}
\begin{array}{clcr}
			h_i=\frac1{\sqrt 2}(f_i+f_j)\\
			h_j=\frac1{\sqrt 2}(f_i-f_j)\\
		\end{array}
\end{equation}
 and $h_k=f_k$ if $k\neq i, j$.

 A straight forward computation gives the following for any $i\neq j$ and $\mathfrak{u}\in U(n)$
\begin{equation}\label{eq-uv1}
  F(ij)(\mathfrak{u})+F(ij)(\mathfrak{u_0}(ij)\mathfrak{u})=\frac12\lf(G(i)(\mathfrak{v_0}(ij)\mathfrak{u})+G(j)(\mathfrak{v_0}(ij)\mathfrak{u})\ri).
\end{equation}
  Thus for $i\neq j$ we have
\begin{equation}\label{eq-integral-1}
\begin{split}
  4\int_{U(n)}F(ij)(\mathfrak{u})d\nu(\mathfrak{u})=&2\int_{U(n)}
  \lf(F(ij)(\mathfrak{u})+F(ij)(\mathfrak{u_0}\mathfrak{u})\ri)d\nu(\mathfrak{u})\\
  =&\int_{U(n)}\lf(G(i)(\mathfrak{v_0}(ij)\mathfrak{u})+G(j)(\mathfrak{v_0}(ij)\mathfrak{u})\ri)d\nu(\mathfrak{u})\\
=&\int_{U(n)}\lf(G(i)(\mathfrak{u})+G(j)(\mathfrak{u})\ri)d\nu(\mathfrak{u}).
\end{split}
\end{equation}
On the other hand, for $i\neq j$ if we let $\mathfrak{w_0}\in U(n)$ be such that for any $\a\in \mathcal{F}$, the frame $\mathfrak{w_0}\a\in \mathcal{F}$ is simply obtained by switching the $i$ and $j$th elements in $\a$, then $G(i)(\mathfrak{u})=G(j)(\mathfrak{w_0}\mathfrak{u})$ and hence
$$
\int_{U(n)}G(i)(\mathfrak{u})d\nu(\mathfrak{u})=\int_{U(n)}G(j)(\mathfrak{u})d\nu(\mathfrak{u}),\
$$ and \eqref{eq-integral-1} thus implies:
\begin{equation}\label{eq-integral-2}
   2\int_{U(n)}F(ij) (\mathfrak{u})d\nu(\mathfrak{u})=\int_{U(n)} G(i)(\mathfrak{u})d\nu(\mathfrak{u}) =\int_{U(n)} G(1)(\mathfrak{u})d\nu(\mathfrak{u})=K(p)
\end{equation}
 for some $K(p)$ depending only on $p$.  In particular, the LHS does not depend on $i, j$.  This establishes the Claim.

 For any $\mathfrak{u}\in U(n)$  we may write

\begin{equation}\label{eq-scalarcurvature-1}
 S(p)=\sum_{i\neq j} F(ij)(\mathfrak{u})+\sum_{i}G(i)(\mathfrak{u}).
\end{equation}

Hence
 \begin{equation}\label{eq-scalarcurvature-2}
 \begin{split}
  S(p)=&\int_{U(n)}\lf(\sum_{i\neq j}F(ij)(\mathfrak{u}) +\sum_{i}G(i)(\mathfrak{u})\ri)d\nu(\mathfrak{u})\\
  =&
   \frac{n(n+1)}2 K(p)
  \end{split}
 \end{equation}
The equality in \eqref{eq-scalarcurvature-2} may be interpreted as saying that on any \K manifold
$S(p)$ is
 either an average of holomorphic sectional curvatures at $p$, a
known
fact from \cite{B} (also see \cite[p189]{Zheng}), or an average of orthogonal holomorphic
bisectional curvatures at $p$.  In particular, if either the holomorphic
sectional curvatures or orthogonal bisectional curvatures are positive
(nonnegative) respectively at $p$, then $S(p)$ is automatically positive
(nonnegative).

Now let  $a_{ij}$ be as in part (a) of the lemma.   Then
\begin{equation}\label{e9}
\begin{split}
  \int_{U(n)}\sum_{i, j=1}^n&F(ij)(\mathfrak{u})a_{ij}d\mu(\mathfrak{u})\\
&= \sum_{i\neq j}a_{ij}\int_{U(n)}F(ij)(\mathfrak{u}) d\mu(\mathfrak{u})+\sum_{i}a_{ii}\int_{U(n)}G(i)(\mathfrak{u}) d\mu(\mathfrak{u})\\
&=\frac {K(p)}2 (\sum_{i\neq j}a_{ij}+2\sum_{i}a_{ii}) .\\
\end{split}
\end{equation}
It follows that if \eqref{e-weaker1} is true for all unitary frames, then $K(p)\ge0$ as $\sum_{i\neq j}a_{ij}+2\sum_{i\neq i}a_{ii}>0$. If in addition \eqref{e-weaker1} is a strict inequality for some unitary frame, then
$$
\sum_{i j}F(ij)(\mathfrak{u})a_{ij}>0
$$ for some $\mathfrak{u}\in U(n)$. By continuity and the fact that \eqref{e-weaker1} is true for all unitary frames, we conclude that $K(p)>0$ and thus $S(p)>0$ by \eqref{eq-scalarcurvature-2}.

The proof of (b) is similar.

\end{proof}

\begin{proof}[Proof of Theorem \ref{t-scalarcurvature}] The first assertion of the theorem follows immediately from Lemma \ref{l-scalar}.   We now prove the other assertions in the Theorem in the case of nonnegative QOBC while the proof nonpositive QOBC is similar.

 We show $S(p)=0$ iff $R(V,\bar V,W,\bar W)=0$ for all unitary pairs $U,V$:  Suppose $R(V,\bar V,W,\bar W)=0$ for all unitary pair.  Then if $K(p)$ is as in the proof Lemma \ref{l-scalar} we have $K(p)=0$, and thus $S(p)=0$ by \eqref{eq-scalarcurvature-2}. 
  Conversely, suppose $S(p)=0$.  By Lemma \ref{l-scalar}(a), \eqref{e1} is then an equality for all
 frames and choices of $\xi$.    Thus for any unitary frame $\{e_1,\dots,e_n\}$, if we let $a_{ij}=R_{i\bar ij\bar j}=a_{ji}$ then the function $f$ of $\xi=(\xi_1,\dots,\xi_n)\in \R^n$ given by
$$
f(\xi)=\sum_{i\neq j}a_{ij}(\xi_i-\xi_j)^2
$$
is identically zero. Hence for all $k\neq l$,
$$
0=\frac{\p^2f}{\p \xi_k\p\xi_l}=a_{kl}+a_{lk}=-2a_{kl}.
$$
Hence $R_{k\bar kl\bar l}=0$. Since $\{e_i\}$ was arbitrary it follows that for any unitary pair $V, W$, $R(V,\bar V,W,\bar W)=0$.
By  \eqref{eq-uv1}, we also have $R(V,\bar V,V,\bar V)+R(W,\bar W,W,\bar W)=0.$

 We now consider the case when $n\geq 3$. If $n\ge3$, then for any unitary pair $V, W$ we can find $U$ such that $U,V,W$ forms a unitary triple.
$$
0=R(U,\bar U,U,\bar U)+  R(W,\bar W,W,\bar W)=-2R(V,\bar V,V,\bar V).
$$
This completes the proof of the Theorem.
\end{proof}

In case $n=2$, the last assertion of the theorem may not be true.  This can be seen by Example 1.2 in \cite{GZ}. Namely, if we let $\Sigma$ be a compact Riemann surface with constant curvature -1 and let $\mathbb{CP}^1$ be the standard sphere. Then $\Sigma\times \mathbb{CP}^1$ has NQOBC. But the scalar curvature is zero everywhere.

\section{Irreducible and reducible manifolds with NQOBC}\label{structure}

We first use Theorem \ref{t-scalarcurvature} to prove that a locally irreducible \K manifold with NQOBC must have positive first Chern class. This generalizes Theorem 2.1 in \cite{GZ}. Note that by Lemma \ref{l-keyfact}, $h^{1,1}(M)=1$ for such a manifold.

\begin{thm}\label{t-chern} Let $(M^n,\omega)$, $n\ge 2$ be compact K\"ahler manifold  with K\"ahler form $\omega$  and with quasi nonnegative OBC at all points. Suppose $h^{1,1}(M)=1$.  Then $c_1(M)=\l[\omega] $ for some $\l> 0$.
 \end{thm}
 \begin{proof} Since $h^{1,1}(M)=1$ we have $c_1(M)=\l[\omega]$ for some $\l$. By Theorem \ref{t-scalarcurvature}, the scalar curvature $S$ of $M$ is nonnegative.
 Hence
 $$
 0\le \int_{M}S\omega^m=\int_{M}\Ric\wedge \omega^{n-1}=\l\, \vol(M)
 $$
 and thus $\l\ge0$. We now show that $l>0$.  Suppose otherwise, and that $\l=0$ and thus $S=0$ everywhere.

  Suppose $n\ge 3$.  Then Theorem  \ref{t-scalarcurvature} implies $M$ is flat which is impossible since $h^{1,1}(M)=1$.
Suppose $n=2$.  At any point we may choose local coordinates $z^i$ such that  $\{\frac{\p}{\p z^i}\}$ are unitary and eigenvectors of $\Ric$. Let $a_i:=R_{i\bar i}$.  Then as $S=0$ we have $a_1=-a_2$ and we further calculate
 \be
 \begin{split}
  \Ric\wedge \Ric=&\lf(\sum_{i=1}^2 a_idz^i\wedge d\bar z^i \ri)\wedge\lf(\sum_{j=1}^2 a_jdz^j\wedge d\bar z^j \ri)\\
  =&2a_1a_2 dz^1\wedge d\bar z^1\wedge dz^2\wedge d\bar z^2 \\
  =&-2a_1^2dz^1\wedge d\bar z^1\wedge dz^2\wedge d\bar z^2.
  \end{split}
  \ee
Since $\l=0$,
$$
\int_{M}\Ric^2=0
$$
which implies $a_1=a_2=0$ everywhere.  On the other hand, Theorem  \ref{t-scalarcurvature} implies $R_{1\bar 12\bar2}=0$ and thus $0=a_i=R_{i\bar i}=R_{i\bar ii\bar i}$ for $i=1, 2$.    Thus $M$ is flat which is impossible is impossible since $h^{1,1}(M)=1.$
This concludes the proof of the Theorem by contradiction.  The above argument for $n=2$ can actually be used for all $n\geq 2$.
 \end{proof}

Next we want to study  the case when $(M, g)$ is possibly reducible.  We will obtain a partial classification of the possible deRham factorizations of the universal cover of $M$ (Theorem \ref{t-reduce}).   Let $\tM$ be the universal cover of $M$ with covering map $\pi$. Let $$\tM=\tM_0\times\tM_1\times\cdots \times \tM_k$$ be the deRham decomposition of $\tM$ where $\tM_0$ is Euclidean with possible zero dimension  and $\tM_i$, $i\ge 1$  (with positive dimension) are irreducible non-flat factors. Then the product of any subcollection of $\tilde M_i$'s still has NQOBC provided the product has dimension at least 2.  By Lemma \ref{l-keyfact},  $c_1(M)$ can be represented by a parallel harmonic (1,1) form $\eta$ so that $\Ric=\eta+\ii\p\bar\p f$ for some smooth function $f$ on $M$.  Then this pullls back to $\t M$ to give
\begin{equation}\label{eq-Ric}
   \wt \Ric=\wt \eta+\ii\p\bar\p \t f
\end{equation}
where $\t f=\pi^*f$ and $\wt \eta=\pi^* \eta$.  Moreover, by Corollary \ref{c1} $\wt \eta$ has the form
$$\wt \eta =\wt\eta_0\times\l_1\tilde \omega_1\times\dots \cdots \times \l_k\tilde\omega_k$$ where $\tilde \omega_i$ is the K\"ahler form of $\tilde M_i$, $\l_i$ are constants and $\wt\eta_0$ is parallel in $\tM_0$.  In the following for a \K manifold $(N,h)$ we use the notation $\Delta_Nu=h^{\ijb}u_{\ijb}$.
\begin{lem}\label{l-irr-1} We have $\wt\eta_0=0$.  Also, if $i\ge1$ and $\dim \tM_i\ge2$ then $\l_i>0$.
\end{lem}
\begin{proof} Suppose $\l_1\le 0$ and $\dim \tM_1\ge2$ say.  Let $$\t p=(\t p_0,\t p_1,\dots,\t p_k)\in \tM_0\times\tM_1\times\cdots \times \tM_k=\tM $$
be such that $f(\pi \t p)=\max_M f$. Consider the function $h(\t q)=\t f (\t p_0,\t q,\dots,\t p_k)$ on $\tM_1$ for $\t q\in \tM_1$. Then $h(\t p_1)=\max_{\t M_1} h$. On the other hand, $\wt\Ric_{(1)}=\l_1\t\omega_1+\ii \p\bar\p h$. By Theorem \ref{t-scalarcurvature}, the scalar curvature of $\tM_1$ is nonnegative because $\dim \tM_1\ge2$ and we conclude that
$$
\Delta_{\tM_1} h\ge 0.
$$
By the strong maximum principle, it follows that $h$ is constant and thus $\wt\Ric_{(1)}=\l_1\omega_1$. By   Theorem \ref{t-scalarcurvature} again  we have $\l_1\ge0$, and so $\l_1=0$ and $\wt\Ric_{(1)}=0$. In particular, the scalar curvature of $\tM_1$ is zero. By Theorem \ref{t-scalarcurvature} and the fact that $\wt\Ric^{(1)}=0$,  we conclude that $\tM_1$ in fact has zero holomorphic sectional curvature everywhere.   This contradicts the fact that $\tM_1$ is nonflat.

To prove that $\wt \eta_0=0$, first note that for the function $\t f_0 (\t q)=\t f(\t q,\t p_1,\dots,\t p_k)$ on $\tM_0$ with $\t p_i$ being fixed:
$$
\Delta_{\tM_0}\t f_0=C
$$
for some constant because $\tM_0$ is flat and $\wt \eta_0$ is parallel. As before, we may conclude that $\t f_0$ is constant. Hence $\wt\eta_0=0$.
\end{proof}

\begin{lem}\label{l-irr-2} Suppose $M$ has NQOBC at $p$ and $M=M_1\times M_2$ with $\dim(M_1)=1$. Let $e_1,\dots,e_n$ be a unitary frame at $p$ such that $e_1$ is tangent to $M_1$ and $e_2,\dots,e_n$ are tangent to $M_2$. Then
$$
\sum_{j=2}^nR(e_2,\bar e_2,e_j,\bar e_j)\ge-R(e_1,\bar e_1,e_1,\bar e_1).
$$
\end{lem}
\begin{proof} Let $f_1=\frac{e_1+e_2}{\sqrt 2}$, $f_2=\frac{e_1-e_2}{\sqrt 2}$, $f_j=e_j$ for $j\ge 3$. Then
$$
R(f_1,\bar f_1,f_2,\bar f_2)=\frac14\lf(R(e_1,\bar e_1,e_1,\bar e_1)+R(e_2,\bar e_2,e_2,\bar e_2)\ri),
$$
and for $j\ge 3$
$$
R(f_1,\bar f_1, f_j,\bar f_j)=\frac12 R(e_2,\bar e_2,e_j,\bar e_j)=R(f_2,\bar f_2, f_j,\bar f_j).
$$
Let $\xi_1=1$, $\xi_2=-1$ and $\xi_j=0$ for $j\ge 3$, then
\be
\begin{split}
0\le &\sum_{i<j}R(f_i,\bar f_i,f_j,\bar f_j)(\xi_i-\xi_j)^2\\
=&4R(f_1,\bar f_1,f_2,\bar f_2)+\sum_{j=3}^nR(f_1,\bar f_1, f_j,\bar f_j)+\sum_{j=3}^nR(f_2,\bar f_2, f_j,\bar f_j)\\
=&R(e_1,\bar e_1,e_1,\bar e_1)+R(e_2,\bar e_2,e_2,\bar e_2)+\sum_{j=3}^nR(e_2,\bar e_2,e_j,\bar e_j).
\end{split}
\ee
From this the result follows.
\end{proof}

\begin{thm}\label{t-reduce} We have the following mutually exclusive cases:

\begin{itemize}
  \item [(i)]  $\dim\tM_0\ge1$ and for each $i\ge 1$: $\tM_i$ is compact with $\l_i>0$ and quasi positive Ricci curvature, i.e., the Ricci curvature of $\tM_i$ is nonnegative and is positive at some point. In particular, if $\dim\tM_i=1$ then $\tM_i$ is biholomorphic to $ \mathbb{CP}^1$.
  \item [(ii)] $\dim\tM_0=0$ and for each $i\ge 1$: $\tM_i$ is compact with $\l_i>0$.  In particular, if $\dim\tM_i=1$ then $\tM_i$ is biholomorphic to $ \mathbb{CP}^1$ and all other factors have nonnegative (positive, respectively) Ricci curvature provided $\tM_i$ has nonpositive (negative, respectively) Gauss curvature somewhere.
  \item [(iii)] $\dim\tM_0=0$ and for some $i\ge 1$: $\tM_i$ is biholomorphic either to $\C$ with $\l_i=0$ or to the unit disk in $\C$ with $\l_i<0$, and all other factors are compact and have $\l_j>0$ with  Ricci curvature bounded below by a positive constant.
\end{itemize}
\end{thm}
\begin{proof}   If $\dim\tM_0\ge1$, by Lemma \ref{l-irr-2} and the fact that $\tM_0$ is flat, we conclude that the Ricci curvature of $\tM_i$ is nonnegative for all $i\ge 1$. Hence $M$ has nonnegative Ricci curvature. By \cite{CheegerGromoll}, $\tM_i$ is compact for all $i\ge 1$. If $\dim \tM_i\ge 2$, then $\l_i>0$ by Lemma \ref{l-irr-1}. If $\dim\tM_i=1$, then it is biholomorphic to $\mathbb{CP}^1$ because it is compact and has nonnegative Gaussian curvature. Since $\Ric=\eta+\ii\p\bar \p f$, the Gauss-Bonnet theorem implies:
$$
4\pi=\int_{\tM_i}K_i=\l_i\,\vol(\tM_i)
$$
where $K_i$ is the Gaussian curvature of $\tM_i$. Hence $\l_i>0$.  That $\t M_i$ has nonnegative Ricci curvature for each $i \geq 1$ follows from  Lemma \ref{l-irr-2}.  Now let $$\t p=(\t p_0,\t p_1,\dots,\t p_k)\in \tM_0\times\tM_1\times\cdots \times \tM_k=\tM $$
be such that $f(\pi \t p)=\min_M f$. Then at $\t p$,
$$
\wt \Ric\ge \wt \eta.
$$
from which it is easy to see that $\tM_i$ has positive Ricci curvature at $\t p_i$.  Thus $\t M_i$ has quasi positive Ricci curvature for each $i \geq 1$. Hence we are in the situation of case (i) in the Theorem.

  Suppose $\dim\tM_0=0$ and $\l_i>0$ for all $i\ge 1$.  By Yau's theorem on Calabi conjecture \cite{Yau1978}, there is a real-valued function $u$ on $M$ such that $g_{\ijb}+u_{ijb}>0$ is \K metric with Ricci form being $\eta$. Pulling this metric back by $\pi$ gives a \K metric on $\tM$ with Ricci form being $\wt \eta$ which is positive and is bounded away from 0, because $\dim\tM_0=0$ and $\l_i>0$. Hence $\tM$ is compact by Myer's theorem. Using Lemma \ref{l-irr-2}, we are thus in the situation of case (ii) in the Theorem.

  Suppose $\dim\tM_0=0$ and $\l_1\le 0$, say. Then $\dim\tM_1=1$ by Lemma \ref{l-irr-1}. The induced metric on $\tM_1$ is of the form $e^{2\lambda}|dz|^2$ where $|dz|^2$ is the standard Euclidean metric on $\C$ or the unit disk in $\C$. Let $h(\t q)=\t f(\t q,\t p_2,\dots,\t p_k)$. Then the Gaussian curvature of $\tM_1$ satisfies:
  $$
  K_1=\l_1+\Delta_{\tM_1}h=\l_1+\frac12e^{-2\lambda}\Delta_0h
  $$
  by \eqref{eq-Ric} and where $\Delta_0$ is the Euclidean Laplacian. On the other hand, we also have
  $$
  K_1=-\Delta_{\tM_1}\lambda=-\frac12e^{-2\lambda}\Delta_0\lambda
  $$
   Hence
  $$
 \l_1= -\frac12 e^{-2\lambda}\Delta_0\lambda(\lambda+h)
  $$
 which is just the Gaussian curvature of the metric $e^{2\lambda+2h}|dz|^2$.  Since $h$ is bounded, the metric $e^{2\lambda+2h}|dz|^2$ is complete. Hence $\l_1=0$ if and only if $\tM_1$ is biholomorphic to $\C$ and $\l_1<0$ if and only if $\tM_1$ is biholomorphic to the unit disk in $\C$.  In case $\l_1=0$, the Gaussian curvature of $\tM_1$ must be negative somewhere. Otherwise, $\tM_1$ must be flat by the proof of \cite[Theorem 3]{CheegerGromoll}. This is impossible, because $\dim\tM_0=0$. In case $\l_1<0$, it is easy to see that the Gaussian curvature of $\tM_1$ is negative somewhere. Using Lemma \ref{l-irr-2}, we are thus in the situation of case  (iii) in the Theorem.

\end{proof}
 \bibliographystyle{amsplain}

 \end{document}